\documentclass[10pt, reqno]{amsart}
\usepackage{amsmath, amsthm, amscd, amsfonts, amssymb ,graphicx, color}
\usepackage[bookmarksnumbered, colorlinks, plainpages]{hyperref}
\hypersetup{colorlinks=true,linkcolor=red, anchorcolor=green, citecolor=cyan, urlcolor=red, filecolor=magenta, pdftoolbar=true}
\usepackage{romannum}
\usepackage{chngcntr}
\counterwithout{equation}{section}
\newtheorem{thm}{Theorem}[section]
\newtheorem{lem}[thm]{Lemma}

\newtheorem{cor}[thm]{Corollary}
\theoremstyle{definition}

\newtheorem{ex}[thm]{Example}

\theoremstyle{remark}

\numberwithin{equation}{section}
\newcommand \0{\textbf{0}}
\newcommand \s{^{[*]}}
\newcommand \+{^{[\dag]}}

\begin{document}
	\setcounter{page}{1}
	\pagenumbering{arabic}
	
	\title{ Reverse Order Law for Generalized Inverses with Indefinite Hermitian Weights}
	\author[ ]  {K. Kamaraj  $^{1}$, P. Sam Johnson $^{2}$ and Athira Satheesh K $^{3}$ }
	
	\address{$^{1}$Department of Mathematics, University College of Engineering Arni, Thatchur, Arni-632326.
		India.}
	\email{\textcolor[rgb]{0.00,0.00,0.84}{krajkj@yahoo.com}}
	
	\address{$^{2}$Department of Mathematical and Computational Sciences,
		National Institute of Technology Karnataka, Surathkal, Karnataka - 575 025, India.}
	\email{\textcolor[rgb]{0.00,0.00,0.84}{ 
			sam@nitk.edu.in}}

	\address{$^{3}$Department of Mathematical and Computational Sciences,
		National Institute of Technology Karnataka, Surathkal, Karnataka - 575 025, India.}
	\email{\textcolor[rgb]{0.00,0.00,0.84}{athirachandri@gmail.com}}

	\keywords{Moore-Penrose inverse, reverse order law, indefinite inner product space, weighted generalized
		inverse}
	
	\date{\today}

	\begin{abstract}
		{In this paper, necessary and sufficient conditions are given for the existence of Moore-Penrose inverse of a product of two matrices in an indefinite inner product space (IIPS) in which reverse order law holds good. Rank equivalence formulas with respect to IIPS are provided and an open problem is given at the end.}
	\end{abstract}
	\maketitle
	\section{Introduction}
	
	The reverse order law for generalized inverse plays an important role in the 
	theoretic research and numerical computations in many areas, including the 
	singular matrix problems, ill-posed problems, optimization problems, and statistics problems (see, for instance, \cite{benisraelbook, golub-book, gohberg-book, radoj, raobook, sun-book, werner}).  A classical result of Greville \cite{Greville} gives necessary and sufficient conditions for the two term reverse order law for the Moore-Penrose inverse in the Euclidean space. It is known that the reverse order law does not hold for various classes of generalized inverses \cite{Cao,Wang}.  Hence, a significant number of papers treat the sufficient or equivalent conditions such that the reverse order law holds in some sense.  Sun and Wei  established some sufficient and necessary conditions for inverse order rule for weighted generalized inverses with positive definite weights \cite{sunwei, sunweitriple}. The concept of the Moore-Penrose inverse between indefinite inner product spaces has been introduced  and mentioned in \cite{kkkcsmpinv} that if the weights are positive definite, then the weighted generalized inverse and the Moore-Penrose inverse between indefinite inner product spaces are the same. In this paper, we give some necessary and sufficient conditions for the existence of Moore-Penrose inverse of a product of two matrices and to hold reverse order law  in an IIPS. Also, we claim that our results are more general than the existing ones for weighted Moore-Penrose inverse.

	\section{Preliminaries}
	
	We consider matrices on the field $ \mathbb{C}$ of complex numbers and denote  the space of complex matrices of order $ m\times n $ by $ \mathbb{C}^{m\times n}$.  The \textit{range} and the \textit{rank} of $ A\in\mathbb{C}^{m\times n}$ are denoted by $R(A)$ and $ rank(A)$ respectively. The \textit{index} of $A\in \mathbb{C}^{n\times n}$  is the least positive integer $p$ such that \emph{rank}$(A^p)= \emph{rank} (A^{p+1})$ and it is denoted by $ind(A)$.
	
	For a complex square matrix $ A $, we call it  \textit{Hermitian} if $ A=A^* $, where $ A^* $ denotes the adjoint of $A$ with respect to the Hermitian inner product $\langle\centerdot,\centerdot\rangle$ on $\mathbb{C}^n$ (i.e., complex conjugate transpose). Let $N$ be an invertible Hermitian matrix of order $n$.  An \textit{indefinite inner product} in $\mathbb{C}^n$ is defined by an equation $$[x,y]=\langle x,Ny\rangle$$ where $x,y \in \mathbb{C}^n$. Such a matrix $N$ is called a \textit{weight}. A space with an indefinite inner product is called an \textit{indefinite inner product space (IIPS)}. Let $ M $ and $ N $ be weights of order $ m $ and $ n $, respectively. The \textit{$MN$-adjoint} of an $m\times n$ matrix $A$ denoted $A^{[*]}$ is defined by $$A^{[*]}= N^{-1}A^*M.$$
	Sun and Wei  \cite{sunwei} used the terminology \textit{weighted conjugate transpose} for \textit{$ MN$-adjoint}. In an IIPS, by considering the same weights $ M=N $, a complex square matrix $A$ is called \textit{$N$-Hermitian} if $A^{[*]}=A$; it is called  \textit{$N$-range Hermitian}  if $R(A)=R(A^{[*]})$.  If the IIPS is understood from the context, then instead of
	saying that $A$ is $N$-Hermitian ($N$-range Hermitian), we may simply say that $A$ is Hermitian (range Hermitian).

	The $MN$-\textit{Moore-Penrose inverse $A^{[\dag ]}$}   of $A\in\mathbb{C}^{m\times n}$ between IIPSs is defined to be the \textit{unique} solution $X\in\mathbb{C}^{n\times m}$, if it exists, to the equations
	\begin{equation} \label{eqn1}
	AXA=A
	\end{equation}
	\begin{equation} \label{eqn2}
	XAX=X
	\end{equation}
	\begin{equation} \label{eqn3}
	(AX)^{[*]}=AX
	\end{equation}
	\begin{equation} \label{eqn4}
	(XA)^{[*]}=XA.
	\end{equation}
	
	The reference to $MN$ will be dropped when there is no ambiguity and $ A^{[\dagger]}$ will be simply called the Moore-Penrose inverse of $ A $.  If $A$ is invertible, then $A^{[\dag ]} = A^{-1}$.  It is easy to observe that if $M$ and $N$ are the identity matrices, then $A\+ = A^\dagger$, where $A^\dagger$ denotes the usual Moore-Penrose inverse in an Euclidean space. Sun and Wei used the notation $A^{\dag}_{MN}$ for $A^{[\dag ]}$  to emphasize on the weights of positive definite Hermite matrices $ M $ and $ N $. In this case, $A^{\dag}_{MN}$ exists for all matrices $A$ and $A^{\dag}_{MN}=N^{-\frac{1}{2}}(M^{\frac{1}{2}}AN^{-\frac{1}{2}})^\dagger M^{\frac{1}{2}}$ \cite{sunwei}. Unlike the Euclidean case and weighted Moore-Penrose inverse, a matrix need not have a Moore-Penrose inverse between IIPSs \cite{kkkcsmpinv}. The following result gives a necessary and sufficient condition for the existence of Moore-Penrose inverse of a matrix between IIPSs.
	
	\begin{thm}[\cite{kkkcsmpinv}, Theorem 1]
		Let $A\in\mathbb{C}^{m\times n}$. Then $A^{[\dag]}$ exists
		iff $rank(A)=rank(AA^{[*]})=rank(A^{[*]}A)$.
	\end{thm}
	
	For the sake of clarity as well as for easier reference we mention the following properties of Moore-Penrose inverse between IIPSs.
	
	\begin{thm}[\cite{kkkcsmpinv}, Section 4] \label{property}
		Let $A\in \mathbb{C} ^{m\times n}$ be such that $A\+$ exists. Then the following statements hold :
		\begin{enumerate}
			\item[(i)] $A\s=A\s AA\+ =A\+AA\s$.
			\item[(ii)] $(A^{[*]})^{[\dag ]} = (A^{[\dag ]} )^{[*]}$. \label{theorem2.2}
			\item[(iii)] $(AA^{[*]})^{[\dag ]}$ and $(A^{[*]}A)^{[\dag ]}$ exist. In this
			case, $(AA^{[*]})^{[\dag ]} =(A^{[*]})^{[\dag ]}A^{[\dag ]}$ and
			$(A^{[*]}A)^{[\dag ]} =A^{[\dag ]}(A^{[*]})^{[\dag ]}$.
			\item[(iv)] $A^{[\dag ]} = A^{[*]}(AA^{[*]})^{[\dag ]} =(A^{[*]}A)^{[\dag ]} A^{[*]}$.
			\item[(v)] $(AA\s)\+ (AA\s)A=A=(AA\s)(AA\s)\+ A$.
			\item[(vi)] $(AA\s)\+ (AA\s) = (AA\s)(AA\s)\+$.
		\end{enumerate}
	\end{thm}
	
	This section is ended with some known results which will be used in the sequel.
	
	\begin{lem}[\cite{benisraelbook}, p.173]\label{rankthm}
		Let $A$ be a square matrix of order $n$ with $ind(A)=1$. Let $B\in\mathbb{C}^{n\times \ell}$ be a matrix such that $R(AB)\subseteq R(B)$. Then 
		\begin{eqnarray*}
			R(AB)=R(A)\cap R(B).
		\end{eqnarray*}
	\end{lem}
	
	\begin{lem}[\cite{sunweitriple}, Lemma 2.1]\label{prop}
		Let $A$, $B$, $C$ and $D$  be matrices with suitable orders. Then
		\begin{eqnarray*}
			rank\left(\begin{array}{cc} A & AB \\ CA & D
			\end{array}\right) =rank(A)+ rank(D-CAB).
		\end{eqnarray*}
	\end{lem}
	
	\begin{lem} [\cite{tian2001}, Theorem 2.7] \label{rankpq}
		Let $P$ and $Q$ are two idempotent matrices of suitable orders. Then
		\begin{eqnarray*}
			rank(PQ-QP)&=&rank\left(\begin{array}{cc}P\\Q\end{array}\right)+rank\left(\begin{array}{cc} P& Q\end{array}\right)+rank(PQ)\\
			&&+rank(QP)-2rank(P)-2rank(Q).	
		\end{eqnarray*}
	\end{lem}
	
	\begin{lem}[\cite{CARLSON}, Corollary] \label{blockmatrix} 
		Let $ M= \left(\begin{array}{cc}A & B \\ C & D \end{array}\right).$ Then $rank(M)=rank(A)$ if and only if $D-CA\+B=0,$ $N(A)\subseteq N(C)$ and $N((A)^*)\subseteq N(B)^*.$
		
		\end{lem}
		
		\begin{lem}[\cite{Tian2002}, Theorem 1.2]\label{propwithMP} 
		Let $A$, $B$, $C$ and $D$  be matrices with suitable orders. Then
		\begin{eqnarray*}
		rank\left(\begin{array}{cc} A^*AA^* & A^*B \\CA^* & D
		\end{array}\right) = rank(A)+rank(D-CA^\dag B).
		\end{eqnarray*}
		\end{lem}
		
		\section{Reverse Order Law}

		We start the section with examples which illustrate that between IIPSs,
		$(AB)^{[\dag]}$ may not exist although $A^{[\dag]}$ and $B^{[\dag]}$ exist, and even though Moore-Penrose inverses of $A, B$ and $AB$  exist, the reverse order law $ (AB)^{[\dag]}=B^{[\dag]}A^{[\dag]}$ does not hold.

		\begin{ex}
		Let $A=\left(\begin{array}{cc} 1& 1\\ 1& 0
		\end{array}\right)$, $B=\left(\begin{array}{ccc} 0 & 1 \\ 0&
			0\end{array}\right)$  and $M=N=\left(\begin{array}{ccc} 1 & 0 \\ 0&
			-1\end{array}\right)$. Clearly, $ B^{[*]}=\left(\begin{array}{cc} 0 & 0\\ -1 & 0
		\end{array}\right)$ and $(AB)^{[*]}= \left(\begin{array}{cc} 0 & 0\\ -1 & 1
		\end{array}\right)$. Then $A$ is non-singular and $ rank(B)=rank(B  B^{[*]})=rank(B^{[*]}B)$, so both $A^{[\dag]}$ and $B^{[\dag]}$ exist.  Also, $ AB=\left(\begin{array}{cc} 0& 1\\ 0& 1
		\end{array}\right)$ and $ rank(AB)\neq rank((AB)^{[*]}AB)$, hence $(AB)^{[\dag]}$ does not exist.
		\end{ex}

		\begin{ex}
		Let $A=\left(\begin{array}{cc} 1& 2\\ 0& 0
		\end{array}\right)$,  $B=\left(\begin{array}{ccc} 2 & 1 \\ 0&
			0\end{array}\right)$ and $M=N=\left(\begin{array}{ccc} 1 & 0 \\ 0&
			-1\end{array}\right)$. Clearly, $A^{[*]}=\left(\begin{array}{cc} 1& 0\\ -2& 0\end{array}\right)$ and $B^{[*]}=\left(\begin{array}{cc} 2& 0\\ -1& 0\end{array}\right)$. Then
		$AA^{[*]}=\left(\begin{array}{cc} -3& 0\\ 0& 0
		\end{array}\right), A^{[*]}A=\left(\begin{array}{cc} 1& 2\\ -2& -4
		\end{array}\right),\\  BB^{[*]}=\left(\begin{array}{cc} 3& 0\\ 0& 0
		\end{array}\right)$ and $B^{[*]}B=\left(\begin{array}{cc} 4& 2\\ -2& -1
		\end{array}\right).$ Hence $A^{[\dag]}=-\frac{1}{3}\left(\begin{array}{cc} 1& 0\\ -2& 0
		\end{array}\right)$ and $B^{[\dag]}=\frac{1}{3}\left(\begin{array}{cc} 2& 0\\ -1& 0
		\end{array}\right)$. Moreover, $(AB)^{[\dag]}=\frac{1}{3}\left(\begin{array}{cc} 2& 0\\ -1& 0
		\end{array}\right)$ and $B^{[\dag]}A^{[\dag]}=-\frac{1}{9}\left(\begin{array}{cc} 2& 0\\ -1& 0
		\end{array}\right)$. Thus $ (AB)^{[\dag]}\neq B^{[\dag]}A^{[\dag]}$.
		\end{ex}

		Motivated by the above examples, we show  when the reverse order law holds good in an indefinite inner product space. Before presenting the main results, we collect some basic results.
		
		\begin{lem}\label{rankabcd}
		Let $A$, $B$, $C$ and $D$ be matrices with suitable orders. If 
		\begin{eqnarray*}
		rank\left(\begin{array}{cc} A & B \\ C & D
		\end{array}\right)=rank(A)=rank(B)=rank(C),
		\end{eqnarray*}
		then $rank(A)=rank(D)$.
		\end{lem}
		\begin{proof}
		Let $M=\left(\begin{array}{cc} A & B \\ C & D
		\end{array}\right)$. It is given that $rank(M)=rank(A)$. Then by Lemma \ref{blockmatrix}, we have $D-CA^\dag B=0$, $R(C^{*})\subseteq R(A^{*})$ and $R(B)\subseteq R(A).$
		This implies $ D=CA^\dag B\implies rank(D)\le rank(C)=rank(A)$.\\
		Now to prove the reverse inequality, $rank(A)=rank(B)=rank(C)\implies R(A)=R(B)$ and $R(A^{*})=R(C^{*})$.
		Since $D=CA^\dag B$ we get $ C^\dag D=C^\dag CA^\dag B=C^\dag{(C^\dag)}^\dag A^\dag B$.
		Now,
		$R(A^{*})=R(C^{*})\implies  R(A^\dag)=R(C^\dag)$. Using the fact that if $R(E)\subseteq R(F)$ then $FF^\dag E=E$, we get $C^\dag{(C^\dag)}^\dag A^\dag=A^\dag. $ This implies $C^\dag D=A^\dag B \implies C^\dag DB^\dag=A^\dag BB^\dag=A^\dag{(B^\dag)}^\dag B^\dag$.
		Now, $R(A)=R(B) \implies R((A^\dag)^{*})=R((B^\dag)^{*})$ and using the fact that if $R(E^{*})\subseteq R(F^{*})$ then $EF^\dag F=E$ we get, $ C^\dag DB^\dag=A^\dag \implies rank(A^\dag)\le rank(D)\implies rank(A)\le rank(D).$ This completes the proof.
		\end{proof}	
		Next we prove the indefinite version of Lemma \ref{propwithMP}.
		\begin{thm}\label{rankequivalance}
		
		Let $A$, $B$, $C$ and $D$  be matrices with suitable orders. If $A\+$ exists, then
		
		\begin{eqnarray*}rank\left(\begin{array}{cc} A\s AA\s & A\s B \\ CA\s & D
		\end{array}\right) 
		&&= rank\left(\begin{array}{cc} D &  CA\s \\ A\s B & A\s AA\s
		\end{array}\right)\\ 
		&&= rank(A)+rank(D-CA\+ B).
		\end{eqnarray*}
		\end{thm}
		\begin{proof}
		By Theorem \ref{property}	we can easily verify the following relations.
		
		\begin{eqnarray*}\left(\begin{array}{cc} (A\+)\s  & 0 \\ 0& I
		\end{array}\right) \left(\begin{array}{cc} A\s AA\s & A\s B \\ CA\s & D
		\end{array}\right) \left(\begin{array}{cc} (A\+)\s  & 0 \\ 0& I
		\end{array}\right)  = \left(\begin{array}{cc} A  & AA\+ B \\ CA\+ A & D
		\end{array}\right)	\end{eqnarray*}
		
		and
		
		\begin{eqnarray*}\left(\begin{array}{cc} A\s  & 0 \\ 0& I
		\end{array}\right)\left(\begin{array}{cc} A  & AA\+ B \\ CA\+ A& D
		\end{array}\right) \left(\begin{array}{cc} A\s  & 0 \\ 0& I
		\end{array}\right) = \left(\begin{array}{cc} A\s AA\s & A\s B \\ CA\s & D
		\end{array}\right).\end{eqnarray*}
		
		Thus
		
		\begin{eqnarray*}rank\left(\begin{array}{cc} A\s AA\s & A\s B \\ CA\s & D
		\end{array}\right) 
		&&=rank \left(\begin{array}{cc} A  & AA\+ B \\ CA\+ A & D \end{array}\right)\\
		&&= rank(A)+rank(D - CA\+ AA\+ B) \text{  (by Lemma \ref{prop})} \\ 
		&&= rank(A)+rank(D - CA\+ B).
		\end{eqnarray*}\\
		Similarly we can prove the other equality.
		
		\end{proof}
		
		It is known in the Euclidean case that the single expression $R(A^*ABB^*)=R(BB^*A^*A)$ is a necessary and sufficient condition for the reverse order law to hold (\cite{benisraelbook},  p.161). This condition was later shown (\cite{hart}, p.231) to hold in a more general setting. The main result and its proof closely follow those of Greville \cite{Greville}. 
		\begin{thm}\label{reverseequalcondn} Let $A\in\mathbb{C}^{m\times n}$ and $B\in\mathbb{C}^{n\times \ell}$.
		If $A^{[\dag]}$ and $B^{[\dag]}$ exist, then the following are
		equivalent:
		\begin{enumerate}
		\item[(i)] $A^{[*]}ABB^{[*]}$ is range Hermitian.
		\item[(ii)] $R(A^{[*]}AB)\subseteq R(B)$ and $R(BB^{[*]}A^{[*]})\subseteq
		R(A^{[*]})$.
		\item[(iii)] $BB^{[\dag]}A^{[*]}A$ and $A^{[\dag]}ABB^{[*]}$ are range
		Hermitian.
		\item [(iv)] $BB^{[\dag]}A^{[*]}AB= A^{[*]}AB$ and $A^{[\dag]}ABB^{[*]}A^{[*]}=BB^{[*]}A^{[*]}$.
		
		\end{enumerate}
		\end{thm}
		\begin{proof} \underline{ $(i)\Rightarrow (ii)$ :}
		As  $B=BB^{[\dag]}B=B B^{[*]} (B^{[\dag]})^{[*]}$, we have
		$R(A\s ABB\s)=R(A\s AB)$. Suppose that $A^{[*]}ABB\s$ is range Hermitian. Then
		$$R(A\s AB)=R(A\s ABB\s)=R(BB\s A\s A)\subseteq R(B).$$ The second part follows similarly.
		
		\noindent \underline{$(ii)\Rightarrow(i)$:}
		Let $C=A\s ABB\s $. Then $C(B\+)\s = A\s AB.$ Hence $R(C)=R(A\s ABB\s ) \subseteq R(A\s AB)=R(C(B\+ )\s ) \subseteq R(C)$. Thus $R(C)=R(A\s AB)$. Similarly $R(C\s )=R(BB\s A\s). $ Thus $A\s ABB\s $ is range Hermitian iff $R(A\s AB)=R(BB\s A\s)$. Suppose $R(A\s AB)\subseteq R(B)$. It is a well-known fact that $ind(A\s A)=1.$ Thus by Lemma \ref{rankthm}, $R(A\s AB)=R(A\s A)\cap R(B)=R(A\s )\cap R(B)$. On the other hand, again by Lemma \ref{rankthm}, $R(BB\s A\s )\subseteq R(A\s )$
		and $ind(BB\s )=1$ give $R(BB\s A\s )=R(BB\s )\cap R(A\s )=R(B)\cap R(A\s ).$ Thus $R(BB\s A\s  )=R(A\s AB). $
		Therefore, $A\s ABB\s $ is range Hermitian.
		
		\noindent \underline{$(ii)\Leftrightarrow(iv)$:} Straight forward.
		
		\noindent \underline{ $(ii)\Rightarrow (iii)$ :}
		Suppose that $R(A\s AB) \subseteq R(B)$.
		As $R(A\s ABB\+)\subseteq R(A\s AB) \subseteq R(B)$, we get
		$A\s ABB\+=BB\+A\s ABB\+$ and hence it can be shown that $$R((BB\+ A\s A)\s)=
		R(A\s ABB\+)= R(BB\+A\s A).$$
		Thus $BB\+ A\s A$ is range Hermitian. In a similar way, using the inclusion relation $$R(BB^{[*]}A^{[*]})\subseteq
		R(A^{[*]}),$$ we can prove that $A^{[\dag]}ABB^{[*]}$ is also range Hermitian.

		\noindent \underline{$(iii) \Rightarrow (ii)$ :}
		Suppose $BB\+ A\s A$ is range Hermitian. Then $R(BB^{\+}A\s A)=R(A\s A BB\+$). It  clear that $R(A\s AB)=R(A\s ABB\+ B)\subseteq R(A\s ABB\+)=R(BB\+ A\s A) \subseteq R(B)$. Thus $BB^{[\dag]}A^{[*]}AB= A^{[*]}AB$. Similarly, we can prove $A^{[\dag]}ABB^{[*]}A^{[*]}=BB^{[*]}A^{[*]}$.

		\end{proof}
		\begin{thm} \label{reverseequalcondn2}
		Let $A\in \mathbb{C}^{m\times n}$, $B\in \mathbb{C}^{n\times \ell}$   and $D=AB$. If $A^{[\dag]}$ and $B^{[\dag]}$ exist, then the following are
		equivalent:
		\begin{enumerate}
		\item [(i)]  $rank\left(\begin{array}{cc} D & AA^{[*]}D \\ DB^{[*]}B &
			DD^{[*]}D \end{array}\right) =rank(D)$, where $D=AB$.
		\item [(ii)] $(AB)^{[\dag]}$ exists and
		$(AB)^{[\dag]}=B^{[\dag]}A^{[\dag]}$.
		\end{enumerate}
		
		\end{thm}
		\begin{proof}

		\noindent \underline{$(i)\Rightarrow(ii)$ :}
		First we prove the existence of Moore-Penrose inverse of $AB$. For that, let $E=AA\s D$. It is easy to observe that $D=AB=(AA\s )\+(AA\s) AB=(AA\s)\+ E$. Thus $rank(D)=rank((AA\s)\+ E)\leq rank(E)=rank(AA\s D)\leq rank(D)$.
		It shows that $rank(D)=rank(AA\s D)$. Similarly, we can prove that $rank(D)=rank(DB\s B)$.
		
		Suppose $rank\left(\begin{array}{cc} D & AA^{[*]}D \\ DB^{[*]}B &
			DD^{[*]}D \end{array}\right) =rank(D)$. Then  $rank(D)=rank(DD\s D)$ by Lemma \ref{rankabcd}. It concludes that $rank(D)=rank(DD\s ) = rank(D\s D)$. Thus $(AB)\+$ exists.	
		By Theorem \ref{rankequivalance},
		\begin{eqnarray*}
		rank\left(\begin{array}{cc} D & AA\s D \\ DB\s B& DD\s D
		\end{array}\right) &&= rank(D\s)+rank(D-AA\s (D\s)\+ B\s B)\\&& = rank(D) + rank(D\s - B\s BD\+ AA\s) .
		\end{eqnarray*}
		Hence, by the assumption $rank(D\s - B\s BD\+ AA\s)  = 0$. Thus $D\s = B\s BD\+ AA\s$. Pre-multiplying by $(B\s B)\+$ and post-multiplying by $(AA\s)\+$ we get 
		\begin{equation*}
		(B\s B)\+ B\s A\s (AA\s )\+ = (B\s B)\+ B\s BD\+ AA\s (AA\s)\+ . 
		\end{equation*}
		By Theorem \ref{theorem2.2} (v) and (vi),
		\begin{eqnarray*}
		B\+ A\+ &&= (B\s B)\+ B\s BB\s A\s (DD\s)\+ AA\s (AA\s)\+ \\&&= B\s A\s (DD\s)\+ AA\s (AA\s)\+\\ &&= D\+ AA\s (AA\s)\+ = (D\s D)\+ D\s AA\s (AA\s)\+ \\&&= D\+ = (AB)\+.
		\end{eqnarray*}
		
		\noindent \underline{$(ii)\Rightarrow(i)$ :}
		By Theorem \ref{rankequivalance},	\begin{eqnarray*}
		rank\left(\begin{array}{cc} D & AA\s D \\ DB\s B& DD\s D
		\end{array}\right) &&= rank(D\s)+rank(D-AA\s(D\s)\+ B\s B)\\&& = rank(D) + rank(D - AA\s(B\+A\+)\s B\s B)\\&& = rank(D)+rank(D - AA\s (A\+)\s(B\+)\s  B\s B) \\&& =  rank(D) +rank(D - AB)\\&& = rank(D) .
		\end{eqnarray*}
		\end{proof}
		
		\begin{thm}\label{reverseequalcondn3}
		Let $A\in \mathbb{C}^{m\times n}$  and $B\in \mathbb{C}^{n\times \ell}$ such that $A^{[\dag]}$ and $B^{[\dag]}$ exist. If any one of the conditions listed in Theorem \ref{reverseequalcondn} holds, then
		
		\begin{eqnarray*}
		rank\left(\begin{array}{cc} D & A\s A D \\ DB\s B& DD\s D	\end{array}\right) = rank(D) .
		\end{eqnarray*}

		\end{thm}
		
		\begin{proof}
		Suppose that
		$BB^{[\dag]}A^{[*]}AB= A^{[*]}AB$ and $A^{[\dag]}ABB^{[*]}A^{[*]}=BB^{[*]}A^{[*]}$	hold.  Then, we have
		
		\begin{eqnarray*}
		\left(\begin{array}{cc} D & AA^{[*]}D \\ DB^{[*]}B & DD^{[*]}D
		\end{array}\right) &&= \left(\begin{array}{cc} AB & ABB^{[\dag]}A^{[*]}AB \\ ABB^{[*]}A^{[\dag]}AB &
		AB(AB)^{[*]}AB \end{array}\right) \\&&= \left(\begin{array}{cc} D & DB^{[\dag]}A^{[*]}D \\ DB^{[*]}A^{[\dag]}D &
		DD^{[*]}D \end{array}\right).
		\end{eqnarray*}
		By Theorem \ref{rankequivalance},
		\begin{eqnarray*}
		rank\left(\begin{array}{cc} D & AA^{[*]}D \\
		DB^{[*]}B & DD^{[*]}D
		\end{array}\right) &&=rank\left(\begin{array}{cc} D & DB^{[\dag]}A^{[*]}D \\ DB^{[*]}A^{[\dag]}D &
		DD^{[*]}D \end{array}\right)\\&& = rank(D)+ rank(DD^{[*]}D-DB^{[*]}A^{[\dag]}DB\+ A\s D)\\
		&& = rank(D)+ rank(DD^{[*]}D-ABB^{[*]}A^{[\dag]}ABB\+ A\s AB)\\
		&& = rank(D)+ rank(DD^{[*]}D-ABB^{[*]}A^{[\dag]}A A\s AB)\\
		&& = rank(D)+ rank(DD^{[*]}D-ABB^{[*]} A\s AB)\\
		&& =rank(D)+rank(DD^{[*]}D-DD^{[*]}D)\\ && =rank(D).
		\end{eqnarray*}
	
		\end{proof}
		
		\begin{cor}
		Let $A\in \mathbb{C}^{m\times n}$  and $B\in \mathbb{C}^{n\times \ell}$ such that $A^{[\dag]}$ and $B^{[\dag]}$ exist. If any one of the conditions listed in Theorem \ref{reverseequalcondn} holds, then $(AB)^{[\dag]}$ exists and
		$(AB)^{[\dag]}=B^{[\dag]}A^{[\dag]}$.
		\end{cor}
		
		\begin{lem}\label{rankrowab}
		Let $A\in \mathbb{C}^{m\times n}, B\in \mathbb{C}^{m\times \ell}$ and $C\in \mathbb{C}^{\ell\times n}$. Then
		\begin{enumerate}
		\item [(i)]\begin{eqnarray*}
		rank\left(\begin{array}{cc} A & B 
		\end{array}\right) =rank\left(\begin{array}{cc} A\s \\ B\s
		\end{array}\right)
		\end{eqnarray*}
		
		\item [(ii)]\begin{eqnarray*}
		rank\left(\begin{array}{cc} A \\ C 
		\end{array}\right) =rank\left(\begin{array}{cc} A\s & C\s
		\end{array}\right).
		\end{eqnarray*}
		\end{enumerate}
		\end{lem}
		\begin{proof}
		\begin{eqnarray*}
		rank\left(\begin{array}{cc} A\s \\ B\s
		\end{array}\right)=rank\left(\begin{array}{cc} N^{-1}A^{*} M \\ L^{-1}B^{*} M 
		\end{array}\right) &&=
		rank\left(\begin{array}{cc} \left(\begin{array}{cc} N^{-1}A^{*}  \\ L^{-1}B^{*}  
		\end{array}\right) & M 
		\end{array}\right) \\ &&=rank\left(\begin{array}{cc}  N^{-1}A^{*}  \\ L^{-1}B^{*}  
		\end{array}\right)\\ &&=rank\left(\begin{array}{cc} AN^{-1} & BL^{-1}
		\end{array}\right)\\ &&=rank\left(\begin{array}{cc} \left(\begin{array}{cc} A & B 
		\end{array}\right) & \left(\begin{array}{cc} N^{-1} & 0 \\  0 & L^{-1}
		\end{array}\right) 
		\end{array}\right) \\ &&=rank\left(\begin{array}{cc} A & B 
		\end{array}\right).
		\end{eqnarray*}
		Similarly we can prove (ii).
		
		\end{proof}

		\begin{lem}\label{adj_dag}
		Let $A\in \mathbb{C}^{m\times n}$  and $B\in \mathbb{C}^{g\times h}$. If $A^{[\dag]}$ and $B^{[\dag]}$ exist, then
		\begin{enumerate}
		\item [(i)] \begin{eqnarray*}
		\left(\begin{array}{cc} A & 0 \\ 0 & B
		\end{array}\right)\s = \left(\begin{array}{cc} A\s & 0 \\ 0 & B\s
		\end{array}\right)
		\end{eqnarray*}
		
		\item [(ii)] \begin{eqnarray*}
		\left(\begin{array}{cc} A & 0 \\ 0 & B
		\end{array}\right)\+ = \left(\begin{array}{cc} A\+ & 0 \\ 0 & B\+
		\end{array}\right)
		\end{eqnarray*}
		
		\item [(iii)] \begin{eqnarray*}
		\left(\begin{array}{cc} 0 & A \\ B & 0
		\end{array}\right)\+ = \left(\begin{array}{cc} 0 & B\+ \\ A\+ & 0
		\end{array}\right).
		\end{eqnarray*}
		\end{enumerate}
		\end{lem}
		
		\begin{proof}
		\begin{enumerate}
		\item [(i)] Let $
		K=	\left(\begin{array}{cc} M & 0 \\ 0 & G
		\end{array}\right)$ and  $
		L=	\left(\begin{array}{cc} N & 0 \\ 0 & H
		\end{array}\right).$
		Without loss of generality we may assume that
		$$T\s =L^{-1}T^{*} K, \text{where } T= \left(\begin{array}{cc} A & 0 \\ 0 & B
		\end{array}\right).$$
		Then  \begin{eqnarray*}T\s &&=\left(\begin{array}{cc} N^{-1} & 0 \\ 0 & H^{-1} \end{array}\right) \left(\begin{array}{cc} A^{*} & 0 \\ 0 & B^{*}
		\end{array}\right) \left(\begin{array}{cc} M & 0 \\ 0 & G
		\end{array}\right)\\ &&=\left(\begin{array}{cc} N^{-1}A^{*} M & 0 \\0 & H^{-1}B^{*} G 
		\end{array}\right)\\ &&= \left(\begin{array}{cc} A\s & 0 \\ 0 & B\s
		\end{array}\right).
		\end{eqnarray*}
		\item [(ii)] Suppose  $A^{[\dag]}$ and $B^{[\dag]}$ exist.
		$$T\s T =\left(\begin{array}{cc} A\s A & 0 \\ 0 & B\s B
		\end{array}\right) \text{ and } TT\s = \left(\begin{array}{cc} AA\s & 0 \\ 0 & BB\s
		\end{array}\right).$$
		Thus $rank(T\s T)=rank(A)+rank(B)=rank(TT\s )=rank(T),$
		which implies $T\+ $ exists.\\
		Also it is easy to verify that $T\+ =\left(\begin{array}{cc} A\+ & 0 \\ 0 & B\+ \end{array}\right)$ satisfies the Moore-Penrose equations.
		\end{enumerate}
		(iii) is similiar to (ii).
		
		\end{proof}
		
		\begin{thm}\label{rankcondition}
		Let $A,B,C,D,P$ and $Q$ be matrices with suitable orders such that $P\+ $ and $Q\+ $ exist. Then 
		\begin{eqnarray*}
		rank(D-CP\+ AQ\+ B)	=rank\left(\begin{array}{ccc} P\s AQ\s & P\s PP\s & 0 \\ Q\s QQ\s & 0 & Q\s B \\ 0 & CP\s & -D
		\end{array}\right)-rank(P)-rank(Q).
		\end{eqnarray*}
		
		\end{thm}
		
		\begin{proof}
		It is observed that
		\begin{eqnarray*}
		\left(\begin{array}{cc} A & AQ\+ B\\ CP\+ A & D \end{array}\right)&&= \left(\begin{array}{cc} A & 0 \\ 0 & D
		\end{array}\right)+ \left(\begin{array}{cc} A & 0 \\ 0 & C
		\end{array}\right)\left(\begin{array}{cc} 0 & Q\+ \\ P\+ & 0 \end{array}\right) \left(\begin{array}{cc} A & 0 \\ 0 & B
		\end{array}\right)\\ &&= \left(\begin{array}{cc} A & 0 \\ 0 & D
		\end{array}\right)+\left(\begin{array}{cc} A & 0 \\ 0 & C
		\end{array}\right)\left(\begin{array}{cc} 0 & P \\ Q & 0 \end{array}\right)\+ \left(\begin{array}{cc} A & 0 \\ 0 & B
		\end{array}\right) 
		\end{eqnarray*} (by Lemma \ref{adj_dag} (iii)).
		
		Thus by Theorem \ref{rankequivalance},
		\begin{eqnarray*}
		rank\left(\begin{array}{cc} A & AQ\+ B\\ CP\+ A & D \end{array}\right) &=&rank\left(\begin{array}{cc} M\s MM\s  & M\s    \left(\begin{array}{cc} A & 0 \\ 0 & B
		\end{array}\right)    \\ 
		\left(\begin{array}{cc} A & 0 \\ 0 & C
		\end{array}\right)  M\s     &    \left(\begin{array}{cc} -A & 0 \\ 0 & -D
		\end{array}\right)      \end{array}\right) -rank\left(\begin{array}{cc} 0 & P \\ Q & 0 \end{array}\right),
		\end{eqnarray*}
		where $M=\left(\begin{array}{cc} 0 & P \\ Q & 0 \end{array}\right).$
		By Lemma \ref{adj_dag} (ii),
		\begin{eqnarray*}
		rank\left(\begin{array}{cc} A & AQ\+ B\\ CP\+ A & D\end{array}\right)&=&rank\left(\begin{array}{cccc} 0 & Q\s QQ\s & 0 & Q\s B\\ P\s PP\s & 0&P\s A & 0\\ 0& AQ\s & -A &0\\CP\s & 0& 0&-D \end{array}\right)\\
		&&-rank(P)-rank(Q).
		\end{eqnarray*} 
		Also 
		\begin{eqnarray*}
		\left(\begin{array}{cccc} 0& I & P\s &0\\I & 0 & 0 & 0 \\ 0 & 0 & 0 & I \\ 0 & 0 & I & 0\end{array}\right) \left(\begin{array}{cccc} 0 & Q\s QQ\s & 0 & Q\s B\\ P\s PP\s & 0&P\s A & 0\\ 0& AQ\s & -A &0\\CP\s & 0& 0&-D \end{array}\right) \left(\begin{array}{cccc} 0 & I & 0 & 0 \\I & 0 & 0 & 0 \\ Q\s & 0 & 0 & I \\ 0 & 0 & I & 0\end{array}\right)\\=\left(\begin{array}{cccc} P\s AQ\s & P\s PP\s & 0 & 0\\ Q\s QQ\s & 0 & Q\s B & 0\\ 0 & CP\s & -D & 0\\ 0 & 0& 0&-A \end{array}\right).
		\end{eqnarray*} 
		Thus 
		\begin{eqnarray*}
		rank\left(\begin{array}{cc} A & AQ\+ B\\ CP\+A & D \end{array}\right)& =&rank \left(\begin{array}{cccc} P\s AQ\s & P\s PP\s & 0 & 0\\ Q\s QQ\s & 0 & Q\s B & 0\\ 0 & CP\s & -D & 0\\ 0 & 0& 0&-A \end{array}\right)\\
		& &-rank(P)-rank(Q)\\&=&rank\left(\begin{array}{ccc} P\s AQ\s & P\s PP\s & 0 \\ Q\s QQ\s & 0 & Q\s B \\ 0 & CP\s & -D   \end{array}\right)+rank(A)\\
		& &-rank(P)-rank(Q)
		.
		\end{eqnarray*} 
		But by Lemma \ref{prop},
		
		\begin{eqnarray*}
		rank\left(\begin{array}{cc} A & AQ\+ B\\ CP\+A & D \end{array}\right)&&=rank(A)+rank(D-CP\+AQ\+B).
		\end{eqnarray*} 
		Thus
		\begin{eqnarray*}
		rank(D-CP\+AQ\+B)=	rank\left(\begin{array}{ccc} P\s AQ\s & P\s PP\s & 0 \\ Q\s QQ\s & 0 & Q\s B \\ 0 & CP\s & -D  \end{array}\right)-rank(P)-rank(Q).
		\end{eqnarray*}
		\end{proof}
		
		\begin{cor}\label{cor13}
		Let $A\in \mathbb{C}^{m\times n}$  and $B\in \mathbb{C}^{n\times \ell}$ such that $A^{[\dag]}$ and $B^{[\dag]}$ exist. Then
		\begin{eqnarray*}
		rank(AB-ABB\+A\+AB)=	rank\left(\begin{array}{cc} B\s A\s & B\s B  \\ AA\s & AB \end{array}\right)+rank(AB)-rank(A)-rank(B).
		\end{eqnarray*}
		\end{cor}
		\begin{proof}
		Replace $D$ by $AB$, $C$ by $AB$, $P$ by $B,$ $A$ by $I$, $Q$ by $A$ and $B$ by $AB$ in Theorem \ref{rankcondition}, we get
		\begin{eqnarray*}
		rank(AB-ABB\+A\+AB)
		&=&	rank\left(\begin{array}{ccc} B\s A\s & B\s BB\s & 0 \\ A\s AA\s& 0 & A\s AB\\ 0 & ABB\s &-AB \end{array}\right)\\
		& &-rank(A)-rank(B).
		\end{eqnarray*}
		Also,
		\begin{eqnarray*}
		\left(\begin{array}{ccc} I & 0 & 0 \\ 0& I & A\s \\ 0 & 0 & I \end{array}\right)\left(\begin{array}{ccc} B\s A\s & B\s BB\s & 0 \\ A\s AA\s & 0 & A\s AB\\ 0 & ABB\s &-AB \end{array}\right)\left(\begin{array}{ccc} I & 0 & 0 \\ 0& I & 0 \\ 0 & B\s & I \end{array}\right)\\=\left(\begin{array}{ccc} B\s A\s & B\s BB\s & 0 \\ A\s AA\s & A\s ABB\s & 0\\ 0 & 0 &-AB \end{array}\right)
		\end{eqnarray*}
		Therefore
		\begin{eqnarray*}
		rank(AB-ABB\+ A\+ AB)
		&=&rank\left(\begin{array}{cc} B\s A\s & B\s BB\s  \\ A\s AA\s & A\s ABB\s \end{array}\right)+rank(AB)\\
		& &-rank(A)-rank(B).
		\end{eqnarray*}
		Moreover, by observing the following facts
		\begin{eqnarray*}
		\left(\begin{array}{cc} I & 0 \\ 0 & {A\+}\s
		\end{array}\right)\left(\begin{array}{cc} B\s A\s & B\s BB\s  \\A\s AA\s & A\s ABB\s \end{array}\right) \left(\begin{array}{cc} I & 0 \\ 0 & {B\+} \s
		\end{array}\right)=\left(\begin{array}{cc} B\s A\s & B\s B \\ AA\s & AB \end{array}\right) 
		\end{eqnarray*}
		and
		\begin{eqnarray*}
		\left(\begin{array}{cc} I & 0 \\ 0 & A\s
		\end{array}\right)\left(\begin{array}{cc} B\s A\s & B\s B \\ AA\s & AB \end{array}\right) \left(\begin{array}{cc} I & 0 \\ 0 & B \s
		\end{array}\right)=\left(\begin{array}{cc} B\s A\s & B\s BB\s \\ A\s AA\s & A\s ABB\s \end{array}\right), 
		\end{eqnarray*}
		we have \begin{eqnarray*}
		rank\left(\begin{array}{cc} B\s A\s & B\s B \\ AA\s & AB \end{array}\right) =rank\left(\begin{array}{cc} B\s A\s & B\s BB\s \\ A\s AA\s & A\s ABB\s \end{array}\right). 
		\end{eqnarray*}
		Thus
		\begin{eqnarray*}
		rank(AB-ABB\+A\+AB)=	rank\left(\begin{array}{cc} B\s A\s & B\s B  \\ AA\s & AB \end{array}\right)+rank(AB)-rank(A)-rank(B).
		\end{eqnarray*}
		\end{proof}
		
		\begin{lem}\label{hermitian}
		
		Let $P$ and $Q$  be two N-Hermitian idempotent matrices of suitable orders. Then
		\begin{eqnarray*}
		rank(PQ-QP)=2\; rank\left(\begin{array}{cc} P  & Q
		\end{array}\right)+2\; rank(PQ)-2\; rank(P)-2\; rank(Q).	
		\end{eqnarray*}
		
		\end{lem}
		\begin{proof}
		Since $P$ and $Q$ are two idempotent matrices, by Lemma \ref{rankpq},
		\begin{eqnarray*}
		rank(PQ-QP)&=&rank\left(\begin{array}{cc} P \\ Q
		\end{array}\right)+rank\left(\begin{array}{cc} P  & Q
		\end{array}\right)+rank(PQ)\\
		&&+rank(QP)-2\; rank(P)-2\; rank(Q).	
		\end{eqnarray*}
		By Lemma \ref{rankrowab},
		\begin{eqnarray*}
		rank(PQ-QP)&=&rank\left(\begin{array}{cc} P\s  & Q\s
		\end{array}\right)+rank\left(\begin{array}{cc} P  & Q
		\end{array}\right)+rank(PQ)\\
		&&+rank(P\s Q\s)-2\; rank(P)-2\; rank(Q).
		\end{eqnarray*}	
		Since $P$ and $Q$ are Hermitian we get,
		\begin{eqnarray*}
		rank(PQ-QP)=2\; rank\left(\begin{array}{cc} P  & Q
		\end{array}\right)+2\; rank(PQ)-2\; rank(P)-2\; rank(Q).	
		\end{eqnarray*}
				\end{proof}
		
			\begin{lem}\label{rankblock}
			If $A^{[\dag]}$ and $B^{[\dag]}$ exist, then	$rank\left(\begin{array}{cc} BB\+ & A\+A 
			\end{array}\right) =rank\left(\begin{array}{cc} B & A\s
			\end{array}\right).$
		\end{lem}
		\begin{proof}
			The	conclusion	may	be	arrived	easily	by	using	the	following	two	equations\\
			\begin{eqnarray*}
				\left(\begin{array}{cc} BB\+ & A\+A 
				\end{array}\right)\left(\begin{array}{cc} B & 0 \\0&A\s
				\end{array}\right)&=\left(\begin{array}{cc} B & A\s
				\end{array}\right)	\text{and}\\	\left(\begin{array}{cc} B & A\s
				\end{array}\right)\left(\begin{array}{cc} B\+ & 0 \\0&(A\+)\s
				\end{array}\right)&=\left(\begin{array}{cc} BB\+ & A\+A 
				\end{array}\right).
			\end{eqnarray*}
		
		\end{proof}
			
		\begin{thm}\label{thm15}
		Let $A\in \mathbb{C}^{m\times n}$  and $B\in \mathbb{C}^{n\times \ell}$. If $A^{[\dag]}$ and $B^{[\dag]}$ exist, then $$rank(BB\+A\+A-A\+ABB\+)=2\; rank\left(\begin{array}{cc} A\s  & B
		\end{array}\right)+2\; rank(AB)-2\; rank(A)-2\; rank(B).$$
		\end{thm}
		\begin{proof}
		Clearly $A\+ A$ and $BB\+$ are Hermitian and idempotent, then by Lemma \ref{hermitian},
		\begin{eqnarray*}
		rank(BB\+ A\+ A-A\+ ABB\+)
		& &=  2\; rank\left(\begin{array}{cc} BB\+  & A\+ A
		\end{array}\right)+2\; rank(BB\+ A\+ A)\\
		& &-2\; rank(BB\+)-2\; rank(A\+ A)\\
		& &=2\; rank\left(\begin{array}{cc} B  & A\s
		\end{array}\right)+2\; rank(AB)\\
		& &-2\; rank(B)-2\; rank(A) \ \ (\text{by  Lemma  \ref{rankblock}}).
		\end{eqnarray*}

		Also, 
		\begin{eqnarray*}
		rank(BB\+ A\+A)& &=rank(B(B\s B)\+ B\s A\s (AA\s)\+ A)\\& &\le rank(B\s A\s)
		=rank(AB)\\ \text{ and } rank(AB)
		=rank(B\s A\s)
		& &=rank(B\s BB\+ A\+ AA\s )\\	& & \le rank(BB\+ A\+A).\end{eqnarray*}
		Thus $rank(AB)= rank(BB\+ A\+A).$
		
		\end{proof}

		\begin{thm}\label{rankequlity}
		Let $A\in \mathbb{C}^{m\times n}$  and $B\in \mathbb{C}^{n\times \ell}$ such that $A^{[\dag]}$ and $B^{[\dag]}$ exist. If
		\begin{eqnarray*}
		rank\left(\begin{array}{cc} B\s A\s & B\s B  \\ AA\s & AB \end{array}\right)=	rank\left(\begin{array}{cc}  A\s & B \end{array}\right),
		\end{eqnarray*}
		then $(AB)\+ =B\+A\+$ is equivalent to any one of the conditions given in Theorem \ref{reverseequalcondn}.
		\end{thm}
		\begin{proof}Since
		\begin{eqnarray*}
		rank\left(\begin{array}{cc} B\s A\s & B\s B  \\ AA\s & AB \end{array}\right)=	rank\left(\begin{array}{cc}  A\s & B \end{array}\right),
		\end{eqnarray*}
		by Theorem \ref{thm15} and Corollary \ref{cor13},  
		$$2\; rank(AB-ABB\+A\+AB)=rank(BB\+A\+A-A\+ABB\+).$$
		Thus if $(AB)\+ =B\+A\+$, then $BB\+A\+A=A\+ABB\+$. Therefore
		$$BB\s A\s =BB\+A\+ABB\s B\s A\s =A\+ A BB\+BB\s A\s =A\+ABB\s A\s.$$
		Similarly, we prove $BB\+A\s AB =A\s AB$. 
		Thus we obtain condition (iv) of Theorem \ref{reverseequalcondn}.
		
		\end{proof}
		\section{An Open problem}
		We can observe from  Theorem \ref{rankequlity} that $(AB)\+ =B\+A\+$ is equivalent to any one of the conditions given in  Theorem \ref{reverseequalcondn} by assuming the rank equality $$	rank\left(\begin{array}{cc} B\s A\s & B\s B  \\ AA\s & AB \end{array}\right)=	rank\left(\begin{array}{cc}  A\s & B \end{array}\right).$$
		But in the Euclidean case such a rank equality assumption is not required. Thus it is an open question for giving proof for Theorem \ref{rankequlity} without assuming the rank equality, or finding a counter example.

		\end{document}